\documentclass[12pt]{amsart}

\usepackage{color}
\usepackage{enumitem}

\usepackage{amsmath,amssymb,setspace,nicefrac, yhmath, amscd,eucal}
\usepackage[active]{srcltx}
\usepackage[colorlinks, linkcolor=red, citecolor=blue, urlcolor=blue, hypertexnames=true]{hyperref}
\usepackage{amsrefs}
\usepackage{tikz}

\allowdisplaybreaks
\usepackage[all]{xy}

\newtheorem{thm}{Theorem}[section]
\newtheorem{cor}[thm]{Corollary}

\newtheorem{question}[thm]{Question}
\newtheorem{definition}[thm]{Definition}

\newtheorem{remark}[thm]{Remark}

\theoremstyle{definition}

\title[Translation-like actions yield regular maps]{Translation-like actions yield regular maps}
\author{Yongle Jiang}
\address{Department of Mathematics, SUNY at Buffalo, USA, 14260--2900} 

\email{yongleji@buffalo.edu}

\begin{document}
\date{\today}
\begin{abstract}
For finitely generated groups $H$ and $G$, we observe that $H$ admits a translation-like action on $G$ implies there is a regular map, which was introduced in Benjamini, Schramm and Tim\'{a}r's joint paper, from $H$ to $G$. Combining with several known obstructions to the existence of regular maps, we have various applications. For example, we show that the Baumslag-Solitar groups do not admit translation-like actions on the classical lamplighter group. 
\end{abstract}

\maketitle
\section{Introduction}

The concept of translation-like actions was introduced by Whyte to solve a geometric version of the von Neumann conjecture \cite{whyte}. It serves as a geometric analogy of subgroup containment for finitely generated groups: if $H$ is a subgroup of a finitely generated group $G$, then the natural right action of $H$ on $G$ is a translation-like action. Later on, answering a question of Whyte, Seward solved a geometric version of Burnside's problem which is formulated using translation-like actions \cite{seward}. Recently, translation-like actions play an important role in studying the question which finitely generated groups admit a weakly aperiodic shift of finite type \cite{jeandel}. Similar approaches (using notions from geometric group theory) to this question also appeared in \cite{cohen_adv., cp}. 

Despite the above success of application of this concept in various problems, it seems to us that not too much is known on whether a group $H$ admits a translation-like action on another finitely generated group $G$ assuming $H$ is not a subgroup of $G$. For known examples, see \cite{cohen, jeandel, seward, whyte}. The purpose of this note is to observe that translation-like actions yield regular maps, which was introduced in \cite{separation}. 

\begin{thm}[Main theorem]\label{main theorem}
Let $G$ and $H$ be finitely generated groups. If $H$ admits a translation-like action on $G$, then $H\rightarrow_{reg} G$.
\end{thm}

It is well-known that there are several ways to rule out existence of regular maps between spaces: asymptotic dimension, Dirichlet harmonic functions, growth and separation (see \cite[last paragraph in section 1]{separation}). By the above theorem, we also get obstructions to admitting translation-like actions for a pair of groups. We discuss these in Section \ref{applications}. 

\section{Definitions and proof of the main theorem}
 
First, we recall the general definition of translation-like actions following Seward \cite[Definition 1.1]{seward}, but here we use left actions. Note that the original definition is due to Whyte \cite[Definition 6.1]{whyte}.

\begin{definition}[Translation-like actions]\label{def of translation-like actions}  
Let $H$ be a group and let $(X, d)$ be a metric space. A left action $*$ of $H$ on $X$ is \emph{translation-like} if it satisfies the following two conditions:
\begin{itemize}
\item[(i)] The action is free (i.e. $h* x=x$ implies $h=1_H$).
\item[(ii)] For every $h\in H$ the set $\{d(x, h* x)|~x\in X\}$ is bounded.
\end{itemize}
\end{definition} 

Next, we recall the definition of regular maps following \cite{separation}.

\begin{definition}[Regular maps]\label{def of regular maps}
Let $(Z, d_Z), (X, d_X)$ be metric spaces, and let $\kappa<\infty$. A (not necessarily continuous) map $f: Z\to X$ is \emph{$\kappa$-regular} if the following two conditions are satisfied.
\begin{itemize}
\item[(i)] $d_X(f(z_0), f(z_1))\leq \kappa(1+d_Z(z_0, z_1))$ holds for every $z_0, z_1\in Z$;
\item[(ii)] For every open ball $B=B(x_0, 1)$ with radius 1 in $X$, the inverse image $f^{-1}(B)$ can be covered by $\kappa$ open balls of radius 1 in $Z$.
\end{itemize} 
A \emph{regular} map is a map which is $\kappa$-regular for some finite $\kappa$. Write $Z\rightarrow_{reg} X$ if there is a regular map from $Z$ to $X$. 
\end{definition} 
 
Note that as mentioned in \cite[p. 5]{separation}, if there is a quasi-isometry between bounded degree graphs $Z$ and $X$, then there is a regular map from $Z$ to $X$ and also the other direction. Hence being quasi-isometric implies $Z\rightarrow_{reg}X$ and $X\rightarrow_{reg}Z$, but the converse does not hold. 

From now on, we fix a finitely generated group $H$ and take $(X, d)$ to be a finitely generated group $G$ using a left-invariant word length metric $d$ associated to some finite symmetric generating set $T$, i.e. $d_G(x, y):=\ell_G(x^{-1}y)$ for all $x, y\in G$, where $\ell_G$ is the word length associated to $T$. 

If $H\overset{*}{\curvearrowright} G$ is a translation-like action, then we can define a map $L: H\times G\to G$ by setting $L(h, x)=x^{-1}(h*x)$ for all $x\in G, h\in H$.
 
Note that $\overset{*}{\curvearrowright}$ is an action implies that $L(h_1, x)L(h_2, h_1*x)=L(h_2h_1, x)$ for all $h_1, h_2\in H$ and $x\in G$.

Write $c(g, x):=L(g, x)^{-1}$, then $c: H\times G\to G$ is a cocycle in the usual sense, i.e. $c(h_1h_2, x)=c(h_1, h_2*x)c(h_2, x)$ for all $h_1, h_2\in H, x\in G$. And the two conditions in the definition of translation-like action are just the following.

\begin{enumerate}
\item For all $x\in G$, the map $H\ni h\mapsto c(h, x)\in G$ is 1-1.
\label{itm:1}
\item For all $h\in H$, the set $\{c(h, x): x\in G\}$ is bounded, i.e. $\sup_{x\in G}\ell_G(c(h, x))=:\lambda_h<\infty$.\label{itm:2}
\end{enumerate}
 
We are ready to prove our main theorem.

\begin{proof}[Proof of Theorem \ref{main theorem}]
Fix any $x\in G$, define $f: H\to G$ by $f(h):=c(h^{-1}, x)^{-1}$. We claim this map is a regular map.

Fix a symmetric generating set $S$ for $H$ and define $\kappa:=\max_{s\in S}\lambda_s+\# T$, where $\lambda_h:=\sup_{x\in G}\ell_G(c(h, x))$ for all $h\in H$.

Note that the second condition in Definition \ref{def of regular maps} is clear since $h\mapsto c(h, x)$ is 1-1. We are left to check the first condition holds.
\begin{eqnarray*}
&&d_G(f(h_1), f(h_2))\\
&=&d_G(c(h_1^{-1}, x)^{-1}, c(h_2^{-1}, x)^{-1})\\
&=&\ell_G(c(h_1^{-1}, x)c(h_2^{-1}, x)^{-1})\\
&=&\ell_G(c(h_1^{-1}, x)c(h_2, h_2^{-1}*x))\\
&=&\ell_G(c(h_1^{-1}h_2, h_2^{-1}*x))\\
&=&\ell_G(c(s_1\cdots s_k, h_2^{-1}*x)),\mbox{~write $h_1^{-1}h_2=s_1\cdots s_k$, where $k=\ell_H(h_1^{-1}h_2), s_i\in S.$}\\
&=&\ell_G(c(s_1, s_2\cdots s_kh_2^{-1}*x)\cdots c(s_k, h_2^{-1}*x))\\
&\leq &\sum_{i=1}^{k}\ell_G(c(s_i, s_{i+1}\cdots s_kh_2^{-1}*x))\\
&\leq &\kappa \ell_H(h_1^{-1}h_2)=\kappa d_H(h_1, h_2).
\end{eqnarray*}
\end{proof}
\begin{remark}
Similar calculation was already used in \cite{kra, lixin}. In fact, there is a connection between translation-like actions and continuous orbit equivalence theory, which may be worth exploring further.
\end{remark}
\section{Applications}\label{applications}
As mentioned in the introduction, we use the obstruction to existence of regular maps, i.e. asymptotic dimension, which was  first defined by Gromov \cite{gromov} (see \cite{a-dimension} for a nice survey), Dirichlet harmonic functions \cite{harmonic}, growth and separation \cite{separation}, to deduce results on non-existence of translation-like actions. Conversely, we may deduce results on the existence of regular maps between two groups. The above process would generate many (counter)examples, we just list a few of them here and we refer the readers to the above papers for definitions of the above invariants.

\begin{cor}
If $G$ is a finitely generated non-amenable group, then $F_2\rightarrow_{reg} G$, where $F_2$ is the non-abelian free group on two generators. 
\end{cor}
\begin{proof}
This is clear by Whyte's solution to the geometric von Neumann conjecture, see \cite{whyte}.
\end{proof}

\begin{cor}\label{app. of seward's thrm}
If $G$ is a finitely generated infinite group, then $\mathbb{Z}\rightarrow_{reg} G$. 
\end{cor}
\begin{proof}
This is clear by Seward's solution to the geometric Burnside's problem, see \cite{seward}.
\end{proof}

\begin{cor}\label{settle cohen's question}
$\mathbb{Z}^d$ does not admit translate-like actions on $F_2$ for all $d\geq 2$. 
\end{cor}
\begin{proof}
By theorem \ref{main theorem}, we just need to show $\mathbb{Z}^d\not\rightarrow_{reg} F_2$ if $d\geq 2$. By \cite[Corollary 3.3]{separation}, $\lim_{n\to \infty}sep_{\mathbb{Z}^d}(n)=\infty$ if $d\geq 2$, while $sep_{F_2}(\cdot)$ is bounded by \cite[Section 2]{separation}. Since the separation function is monotone non-decreasing with respect to regular maps by \cite[Lemma 1.3]{separation}, we deduce $\mathbb{Z}^d\not\rightarrow_{reg}F_2$. 
\end{proof}
Corollary \ref{settle cohen's question} answers a question of Cohen, see \cite[p.4]{cohen}. This question is one initial motivation for us to look at translation-like actions.

\begin{cor}\label{settle jeandel's question}
The Baumslag-Solitar group $BS(m, n):=\langle s, t|~st^ms^{-1}=t^n\rangle$ does not admit translate-like actions on the lamplighter group $(\mathbb{Z}/2\mathbb{Z})\wr\mathbb{Z}$ for any integers $m, n$. 
\end{cor}
\begin{proof}
Case 1: Assume $mn\neq 0$.

It suffices to show $BS(m, n)\not\rightarrow_{reg}(\mathbb{Z}/2\mathbb{Z})\wr\mathbb{Z}$ if $mn\neq 0$.

First, $asdim((\mathbb{Z}/2\mathbb{Z})\wr\mathbb{Z}))=1$ by \cite[Proposition on p.5]{gentimis}. Then note that the asymptotic dimension is monotone non-decreasing under regular maps, see \cite[Section 6]{separation} or just apply \cite[Theorem 29]{a-dimension} in our setting. Therefore, we are left to show $asdim(BS(m, n))\geq 2$ if $mn\neq 0$.

First, $BS(m, n)$ is always an infinite group, hence $asdim(BS(m, n))>0$ by \cite[Lemma 1]{gentimis} or Corollary \ref{app. of seward's thrm}. We are left to show $asdim(BS(m, n))\neq 1$. By \cite[Corollary 1.2]{fuji_whyte} (see also \cite{cohomo. approach, gentimis, filling invariant}), we just need to check $BS(m, n)$ does not contain free group as a subgroup of finite index. This is clear since $BS(m, n)$ is torsion-free when $mn\neq 0$ by \cite{kms} and any torsion-free virtually free groups are free groups by Stallings' work \cite{stallings}.

Case 2: Assume $mn=0$.

$BS(m, n)\cong \mathbb{Z}*(\mathbb{Z}/n\mathbb{Z})$ or $\mathbb{Z}*(\mathbb{Z}/m\mathbb{Z})$. This group contains a free group, then it does not admit a translation-like action on $(\mathbb{Z}/2\mathbb{Z})\wr \mathbb{Z}$ by \cite[Lemma 1.3]{jeandel} and \cite[Theorem 6.2]{whyte}. 
\end{proof}
\begin{remark}
Note that to prove the above corollary, we can directly focus on $BS(1, n)$, which is amenable. Also note that if $m=n=1$, $BS(1, 1)=\mathbb{Z}^2$, then we can also deduce $\mathbb{Z}^2\not\rightarrow_{reg} (\mathbb{Z}/2\mathbb{Z})\wr\mathbb{Z}$ by \cite[Proposition 6.1]{separation}. If $m=0$ or $n=0$, then $BS(m, n)\cong \mathbb{Z}*(\mathbb{Z}/n\mathbb{Z})$ or $\mathbb{Z}*(\mathbb{Z}/m\mathbb{Z})$, this group has asymptotic dimension one by \cite[Theorem 84]{a-dimension}.
\end{remark}
Corollary \ref{settle jeandel's question} answers \cite[Conjecture 3]{jeandel} negatively. It also answers the  geometric Gersten problem stated in \cite{seward} in the negative. The original geometric Gersten problem (stated under a further assumption compared with the one stated in \cite{seward}) was asked by Whyte in \cite[p.107]{whyte}. See \cite{cohen} for more discussion on this conjecture. Note that the lamplighter group is a potential counterexample to the above conjectures was already suggested in \cite{jeandel}. 

We end the paper with the following question:
\begin{question}
Does \cite[Theorem 2.3]{hs} still hold for regular maps (maybe up to some modification of the definition of fat bigons there)?
\end{question}

\textbf{Acknowledgments.} The author would like to thank his advisor  Prof. Hanfeng Li for constant support and illuminating discussion on this paper. He is also grateful to Nhan-Phu Chung and Xin Li for sharing their ideas and unpublished notes related to continuous orbit equivalence theory. He also thanks David Bruce Cohen for very helpful correspondence.

\end{document}